\documentclass{amsart}

\usepackage{amsfonts,amsmath,amssymb,amsthm,mathrsfs,amsxtra}
\usepackage{enumerate,verbatim}
\usepackage{lineno} 
\usepackage[all,2cell,ps]{xy}
\usepackage[pagebackref]{hyperref}

\DeclareMathOperator{\Ext}{Ext}

\DeclareMathOperator{\Hom}{Hom}
\DeclareMathOperator{\Image}{Image}

\DeclareMathOperator{\Ker}{Ker}

\DeclareMathOperator{\reg}{reg}

\DeclareMathOperator{\Tor}{Tor}

\renewcommand{\ge}{\geqslant}
\renewcommand{\le}{\leqslant}

\newcommand{\bN}{\mathbb{N}}
\newcommand{\bZ}{\mathbb{Z}}

\newcommand{\cN}{\mathcal{N}}
\newcommand{\fg}{finitely generated }
\newcommand{\fm}{\mathfrak{m}}

\renewcommand{\iff}{if and only if }
\newcommand{\lra}{\longrightarrow}

\newcommand{\sR}{\mathscr{R}}
\newcommand{\sS}{\mathscr{S}}

\theoremstyle{plain}
\newtheorem{theorem}{Theorem}[section]
\newtheorem{lemma}[theorem]{Lemma}

\newtheorem{corollary}[theorem]{Corollary}

\theoremstyle{definition}

\newtheorem{example}[theorem]{Example}

\newtheorem{question}[theorem]{Question}
\newtheorem{hypothesis}[theorem]{Hypothesis}

\newtheorem{para}[theorem]{}

\theoremstyle{remark}
\newtheorem{remark}[theorem]{Remark}

\numberwithin{equation}{theorem}

\title[An asymptotic bound for Castelnuovo-Mumford regularity]
{An asymptotic bound for Castelnuovo-Mumford regularity of certain Ext modules over graded complete intersection rings}

\author[Dipankar Ghosh]{Dipankar Ghosh}
\address{Chennai Mathematical Institute, H1, SIPCOT IT Park, Siruseri, Kelambakkam, Chennai 603103, Tamil Nadu, India}
\email{dghosh@cmi.ac.in}

\author[Tony J. Puthenpurakal]{Tony J. Puthenpurakal}
\address{Department of Mathematics, Indian Institute of Technology Bombay, Powai, Mumbai 400076, India}
\email{tputhen@math.iitb.ac.in}

\date{June 27, 2018}
\subjclass[2010]{Primary 13D07; Secondary 13D02, 13A30}
\keywords{Complete intersection rings; Multigraded rings and modules; Castelnuovo-Mumford regularity; Ext; Eisenbud operators}

\begin{document}


\begin{abstract}
	Set $ A := Q/({\bf z}) $, where $ Q $ is a polynomial ring over a field, and $ {\bf z} = z_1,\ldots,z_c $ is a homogeneous  $ Q $-regular sequence. Let $ M $ and $ N $ be finitely generated graded $ A $-modules, and $ I $ be a homogeneous ideal of $ A $. We show that
	\begin{enumerate}
		\item 
		$ 
		\reg\left( \Ext_A^{i}(M, I^nN) \right) \le \rho_N(I) \cdot n - f \cdot \left\lfloor \frac{i}{2} \right\rfloor + b \quad \mbox{for all } i, n \ge 0 ,
		$
		\item 
		$ 
		\reg\left( \Ext_A^{i}(M,N/I^nN) \right) \le \rho_N(I) \cdot n - f \cdot \left\lfloor \frac{i}{2} \right\rfloor + b' \quad \mbox{for all } i, n \ge 0,
		$
	\end{enumerate}
	where $ b $ and $ b' $ are some constants, $ f := \min\{ \deg(z_j) : 1 \le j \le c \} $, and $ \rho_N(I) $ is an invariant defined in terms of reduction ideals of $ I $ with respect to $ N $. There are explicit examples which show that these inequalities are sharp.
\end{abstract}

\maketitle

\section{Introduction}\label{sec:intro}
 
 Castelnuovo-Mumford regularity is a kind of universal bound for important invariants such as the maximum degree of the minimal generators of syzygy modules and the maximum non-vanishing degree of the local cohomology modules. One has often tried to obtain upper bounds for the regularity in terms of simpler invariants. In this article, we provide such bounds for Ext modules involving powers of ideals over complete intersection rings.
 
 Let $ Q = K[X_1,\ldots, X_d] $ be a polynomial ring over a field $ K $ with its usual grading, i.e., each $ X_i $ has degree $ 1 $, and let $\fm$ denote the maximal homogeneous ideal of $ Q $. Let $ N $ be a finitely generated non-zero graded $ Q $-module. The Castelnuovo-Mumford regularity of $ N $, denoted by $ \reg(N) $,  is defined to be the least integer $ m $ so that, for every $ j $, the $j$th syzygy of $ N $ is generated in degrees $ \le m + j $.
 
 For a homogeneous ideal $ I $, in $ Q $, Kodiyalam \cite[Thm.~5]{Kod00} and Cutkosky, Herzog and Trung \cite[Thm.~1.1]{CHT99} independently proved that $ \reg(I^n) = a n + b $ for all $ n \gg 0 $, where $ a $ and $ b $ are constants.  The number $ a $ can be determined using a reduction of $ I $.  Kodiyalam also proved that for $ N $ and $ Q $ as above, the regularity of the $ n $th symmetric power of $ N $ and related modules are bounded above by linear functions of $ n $ with leading coefficient at most the maximal degree of a minimal generator of $ N $; see \cite[Cor.~2]{Kod00}.
 
 For many ideals, in \cite[Cor.~4.4]{Gho16a}, the first author showed that for ideals $ I_1,\ldots, I_t $, there exists positive integer $ a $ such that 
 \[
 	\reg(I_1^{n_1}\cdots I_t^{n_t}) \le (n_1 + \cdots + n_t) a \ \quad \text{for all } \ n_1,\ldots,n_t \ge 0.
 \]
 A natural question is whether in the multigraded case, the regularity (for many ideals) is asymptotically linear.
 See \cite[3.1]{BC17} for an example where this fails.
 
 Let $ \mathfrak{a} $ be a homogeneous ideal of $ Q $. Set $ A := Q/\mathfrak{a} $. If $ M $ is a finitely generated graded $ A $-module, then we can simply set $ \reg_A(M) := \reg_Q(M) $. There is also an intrinsic definition for regularity of graded modules over $ A $ which coincides with $ \reg_Q(M) $ (see \ref{reg-over-Q-or-A}). So we simply denote it by $ \reg(M) $.
 
 If $ A $ is a singular ring, then there exists phenomena which does not arise in the polynomial ring case. For instance, given finitely generated graded modules $ M $ and $ N $, it is unknown whether $ \reg\left( \Tor^A_i(M, N) \right) $ and $ \reg \left( \Ext^i_A(M, N) \right) $ are bounded above by linear functions of $ i $. However, over polynomial ring $ Q = K[X_1,\ldots, X_d] $, the following bounds are known:
 \begin{enumerate}
 	\item \cite[Cor.~3.1]{EHU06} If $ \dim(\Tor_1^Q(M, N)) \le 1 $, then
 	\[
 		\reg\big( \Tor_i^Q(M, N) \big) \le i + \reg(M) + \reg(N) \quad \mbox{for every } 0 \le i \le d.
 	\]
 	\item \cite[Thm.~4.6]{CD08} If $ \dim (M \otimes_Q N) \le 1 $, then
 	\[
 		\max_{0 \le i \le d} \left\{ \reg\big( \Ext_Q^i(M,N) \big) + i \right\} = \reg(N) - \mathrm{indeg}(M),
 	\]
 	where $ \mathrm{indeg}(M) := \inf\{ n \in \mathbb{Z} : M_n \neq 0 \} $.
 	\item \cite[Thm.~2.4(2) and 3.5]{CHH11} An upper bound of $ \reg( \Ext_Q^i(M,N) ) + i $ is given in terms of certain invariants of $ M $ and $ N $.
 \end{enumerate}
 On the other hand, if $ I $ is a homogeneous ideal of $ A $, then it follows from the work of Trung and Wang \cite[Thm.~2.2]{TW05} that for every fixed $ i $, the numbers
 \begin{align*}
  \reg \left( \Tor^A_i(M, I^n N) \right), &\quad \reg \left( \Tor^A_i(M, N/I^n N) \right) \\
  \reg \left( \Ext^i_A(M, I^n N) \right) \quad &\mbox{and} \quad \reg \left( \Ext^i_A(M, N/I^n N) \right)
 \end{align*}
 are bounded above by linear functions of $ n $. The purpose of the article is to show over complete intersection rings, what happens when both $ i $ and $ n $ vary. We prove that if $ A = Q/(z_1,\ldots,z_c) $ is a graded complete intersection, and $ I $ is a homogeneous ideal of $ A $, then for every $ l \in \{ 0,1 \} $, we have the following bounds:
 \begin{enumerate}
 	\item $ \reg\left( \Ext_A^{2i+l}(M, I^nN) \right) \le \rho_N(I) \cdot n - f \cdot i + b_{l} $ \quad for all $ i, n \ge 0 $,
 	\item $ \reg\left( \Ext_A^{2i+l}(M,N/I^nN) \right) \le \rho_N(I) \cdot n - f \cdot i + b'_{l} $ \quad for all $ i, n \ge 0 $,
 \end{enumerate}
 where $ b_{l} $ and $ b'_{l} $ are some constants, $ f := \min\{ \deg(z_j) : 1 \le j \le c \} $, and $ \rho_N(I) $ is an invariant defined in terms of reduction ideals of $ I $ with respect to $ N $; see Theorem~\ref{thm:main}. We also give explicit examples which show that these inequalities are sharp.
 
 We now describe in brief the contents of this article. In Section~\ref{sec:preli}, we give notations and preliminaries on Castelnuovo-Mumford regularity. The main technique in this article is to analyze the cohomological operators (due to Eisenbud) in the graded setup, and how these operators provide multigraded module structures on Ext. These structures are described in Section~\ref{sec: module structure on Ext}. We deduce our main results from a theorem over arbitrary trigraded setup which is proved in Section~\ref{sec: linear bounds in multigraded modules}. Our main results are shown in Section~\ref{sec: regularity of Exts}. Finally, in Section~\ref{sec:examples}, we give a few examples.
 
\section{Notations and preliminaries on regularity}\label{sec:preli}
 
 \begin{para}\label{notations}
  Throughout this article, all rings are assumed to be commutative Noetherian rings with identity. We denote the set of all non-negative integers by $ \bN $, and the set of all integers by $ \bZ $. Let $ t $ be a fixed positive integer.
  For each $ 1 \le u \le t $, $ \underline{e}^u $ denotes the $ u $th standard basis element of $ \bZ^t $. We set $ \underline{0} := (0,\ldots,0) \in \bZ^t $.
  Let $ R $ be a $ \bZ^t $-graded ring, and $ L $ be a $ \bZ^t $-graded $ R $-module. By $ L_{\underline{n}} $, we always mean the $ \underline{n} $th graded component of $ L $.
%
An $ \bN^t $-graded ring $ R $ is said to be {\it standard $ \bN^t $-graded} if $ R $ is generated in total degree one, i.e., if $ R = R_{\underline{0}}[R_{\underline{e}^1},\ldots,R_{\underline{e}^t}] $.
\end{para}
 
 \begin{para}\label{reg-coho}
 	Let $A = A_0[x_1,\ldots,x_d]$ be a standard $\mathbb{N}$-graded ring, i.e., $ \deg(A_0) = 0 $, and $ \deg(x_i) = 1 $ for $ 1 \le i \le d $. Set $ A_{+} := \bigoplus_{n \ge 1}A_n $. Let $M$ be a finitely generated $ \bZ $-graded $A$-module. For every $i \ge 0$, suppose $H_{A_{+}}^i(M)$ denotes the $i$th local cohomology module of $M$ with respect to $A_{+}$. We set $ a_i(M) := \max\{ \mu : H_{A_{+}}^i(M)_{\mu} \neq 0 \} $ if $H_{A_{+}}^i(M) \neq 0$, and $a_i(M) := -\infty$ otherwise. The {\it Castelnuovo-Mumford regularity} of $M$ is defined to be $ \reg(M) := \max\left\{ a_i(M) + i : i \ge 0 \right\} $. Note that $ M = 0 $ \iff $ \reg(M) = -\infty $.
 \end{para}
 
 For a given short exact sequence, by considering the corresponding long exact sequence of local cohomology modules, one obtains the following well-known result. 
 
 \begin{lemma}\label{lemma: properties of regularity}
 	Let $A$ be as above. If $0 \rightarrow M' \rightarrow M \rightarrow M'' \rightarrow 0$ is a short exact
 	sequence of finitely generated $\mathbb{Z}$-graded $A$-modules, then we have the following:
 	\begin{enumerate}[{\rm (i)}]
 		\item $\reg(M') \le \max\{ \reg(M), \reg(M'') + 1\}$.
 		\item $\reg(M) \le \max\{ \reg(M'), \reg(M'')\}$.
 		\item $\reg(M'') \le \max\{ \reg(M') - 1, \reg(M)\}$.
 		\item If $ M' $ has finite length, then $\reg(M) = \max\{ \reg(M'), \reg(M'') \}$.
 	\end{enumerate}
 \end{lemma}
 
 \begin{para}\label{reg-syz}
	Let $ Q = K[X_1,\ldots,X_d] $ be a polynomial ring in $ d $ variables over a field $ K $ with the usual grading, i.e., $ \deg(X_i) = 1 $ for $ 1 \le i \le d $. Let $M$ be a finitely generated $ \bZ $-graded $Q$-module. In this setup, $ \reg(M) $ can also be expressed in terms of degrees of homogeneous generators for syzygies of $ M $ over $ Q $. Let
	\begin{equation}\label{reso}
		0 \to F_{d'} \to F_{d'-1} \to \cdots \to F_1 \to F_0 \to 0
	\end{equation}
	be a finite graded free resolution of $ M $ over $ Q $, where $ F_i = \bigoplus_{j \in \mathbb{Z}} Q(-j)^{\beta_{ij}} $ for $ 0 \le i \le d' $. (Note that $ \beta_{ij} = 0 $ for all but finitely many $ j $). It is well known that
	\begin{equation}\label{reg-via-syz}
		\reg(M) \le \max \{ j - i : \beta_{ij} \neq 0, \mbox{ where }  0 \le i \le d' \mbox{ and } j \in \mathbb{Z} \}
	\end{equation}
	with equality if the resolution \eqref{reso} is minimal; see, e.g., \cite[16.3.7]{BS13}.
 \end{para}
 
\begin{para}\label{reg-over-Q-or-A}
	Let $A = K[x_1,\ldots,x_d]$ be a standard $\mathbb{N}$-graded ring over a field $ K $. Set $ A_{+} := \bigoplus_{n \ge 1} A_n $. Let $ Q $ be as in \ref{reg-syz}. Then $ A \cong Q/\mathfrak{a} $ for some homogeneous ideal $ \mathfrak{a} $ of $ Q $. Let $M$ be a finitely generated $ \bZ $-graded $A$-module (hence $ Q $-module too). It is well known that $ H_{Q_+}^i (M) \cong H_{A_+}^i (M) $ for every $ i \ge 0 $; see, e.g., \cite[14.1.7(ii)]{BS13}.
	Hence $ \reg_Q(M) = \reg_A(M) $, where $ \reg_A(M) $ is the regularity of $ M $ as an $ A $-module. So we simply denote this by $ \reg(M) $.
\end{para}

\section{Multigraded module structure on Ext}\label{sec: module structure on Ext}
 
 We work with the following setup.
 
 \begin{hypothesis}\label{setup0}
 	Let $ Q = K[X_1,\ldots,X_d] $ be a polynomial ring over a field $ K $, where $ \deg(X_i) = 1 $ for $ 1 \le i \le d $. Set $ A := Q/({\bf z}) $, where $ {\bf z}  = z_1,\ldots,z_c $ is a homogeneous $ Q $-regular sequence. Then $ A $ is a standard $ \bN $-graded ring. We write $ A = K[x_1,\ldots,x_d] $, where $ x_i $ is the residue of $ X_i $, i.e., $ \deg(x_i) = 1 $ for $ 1 \le i \le d $. Let $ J $ be an ideal of $ A $ generated by homogeneous elements $ y_1,\ldots,y_b $ of degree $ d_1,\ldots,d_b $ respectively. Set $ f_j := \deg(z_j) $ for $ 1 \le j \le c $. Without loss of generality, we may assume that $ f_1 \le \cdots \le f_c $ and $ d_1 \ge \cdots \ge d_b $. Let $ M $ and $ N $ be \fg $ \bZ $-graded $ A $-modules.
 \end{hypothesis}
 
 \begin{para}[Eisenbud operators]\label{para:mod-struc-1}
 	We need Eisenbud operators (\cite[Section~1]{Eis80}) in the graded setup. By a {\it homogeneous homomorphism}, we mean a graded homomorphism of degree zero. Let $ \mathbb{F} : ~ \cdots \rightarrow F_n \rightarrow \cdots \rightarrow F_1\rightarrow F_0 \rightarrow 0 $ be a $ \bZ $-graded free resolution of $M$ over $A$. In view of the construction of the Eisenbud operators, one may choose homogeneous $ A $-module homomorphisms $ t'_j : F_{i+2} \to F_i(-f_j) $ (for every $ i $) corresponding to $ z_j $; see \cite[page~39, (b)]{Eis80}. Note that if $ F_i = A^{b_i} $, then $ F_i(-f) \cong A(-f)^{b_i} $. We set
 	\[
 		\mathbb{F}^{(j)} : ~ \cdots \longrightarrow F_n(-f_j) \longrightarrow \cdots \longrightarrow F_1(-f_j) \longrightarrow F_0(-f_j) \longrightarrow 0
 	\]
 	for $ 1 \le j \le c $. Thus the Eisenbud operators corresponding to $ {\bf z}  = z_1,\ldots,z_c $ are given by $ t'_j : \mathbb{F}(+2) \rightarrow \mathbb{F}^{(j)} $ ($ 1 \le j \le c $), where the complex $ \mathbb{F}(+2) $ is same as $ \mathbb{F} $ but the degree is shifted by $ +2 $, i.e., $ \mathbb{F}(+2)_n = F_{n+2} $ for all $ n $. The homogeneous chain maps $t'_j$ are determined uniquely up to 	homotopy; \cite[1.4]{Eis80}. Therefore the maps
 	\[
 		\Hom_A(t'_j,N) : \Hom_A(\mathbb{F}^{(j)},N) \longrightarrow \Hom_A(\mathbb{F}(+2),N)
 	\]
 	induce well-defined homogeneous $ A $-module homomorphisms
 	\begin{equation}\label{Eis-op}
 		t_j : \Ext_A^i(M,N) \longrightarrow \Ext_A^{i+2}(M,N)(-f_j) \quad \mbox{for all $ i \ge 0 $ and $1 \le j \le c$}
 	\end{equation}
 	It is shown in \cite[1.5]{Eis80} that the chain maps $t'_j$ ($ 1 \le j \le c $) commute up to homotopy. Thus
 	\[
 		\Ext_A^{\star}(M,N) := \bigoplus_{i \ge 0} \Ext_A^i(M,N)
 	\]
 	turns into a graded $\mathscr{T} := A[t_1,\ldots,t_c]$-module, where $\mathscr{T}$ is the graded polynomial ring  over $A$ in the {\it cohomology operators} $t_j$ with $\deg(t_j) = 2$ for $1 \le j \le c$. These structures depend only on ${\bf z}$, are natural in both module arguments and commute with the connecting maps induced by short exact sequences.
 \end{para}
 
 \begin{para}\label{para:Gulliksen}
 	In \cite[Thm.~3.1]{Gul74}, Gulliksen proved that $ \Ext_A^{\star}(M,N) $ is finitely generated
 	over $ A[t_1,\ldots,t_c] $.
 \end{para}
 
 \begin{para}[Bigraded module structure on Ext]\label{para:mod-struc-2}
 	Suppose $\mathscr{R}(J)$ denotes the {\it Rees ring} $\bigoplus_{n \ge 0} J^n X^n$ associated to $ J $. We may consider $\mathscr{R}(J)$ as a subring of the polynomial ring $A[X]$. Let $\mathcal{N} = \bigoplus_{n \ge 0} N_n$ be a graded $\mathscr{R}(J)$-module such that $ N_n $ is a $ \bZ $-graded $ A $-module for every $ n \in \bN $. Let $ y \in J^sX^s $ for some $ s \ge 0 $. Furthermore, assume that $ y $ (as an element of $ J^s \subseteq A $) is homogeneous of degree $ d_y $. Then there are homogeneous $ A $-module homomorphisms $ N_n \stackrel{y\cdot}{\longrightarrow} N_{n+s}(d_y) $ ($ n \ge 0 $) given by multiplication with $ y $. By applying $\Hom_A(\mathbb{F},-)$ on these maps, and using the naturality of the Eisenbud operators $t'_j$, we obtain the following commutative diagram of cochain complexes:
 	\[
 		\xymatrixrowsep{6mm} \xymatrixcolsep{9mm}
 		\xymatrix{
 			\Hom_A\left(\mathbb{F}^{(j)},N_n\right) \ar[d]^{y} \ar[r]^{t'_j} &\Hom_A(\mathbb{F}(+2),N_n) \ar[d]^{y} \\
 			\Hom_A\left(\mathbb{F}^{(j)},N_{n+s}(d_y)\right) \ar[r]^{t'_j} &\Hom_A\big(\mathbb{F}(+2),N_{n+s}(d_y)\big).
 		}
 	\]
 	Taking cohomology, we get the following commutative diagrams of homogeneous $A$-module homomorphisms:
 	\begin{equation}\label{action}
 		\xymatrixrowsep{6mm} \xymatrixcolsep{12mm}
 		\xymatrix{
 			\Ext_A^i(M,N_n) \ar[d]^{y} \ar[r]^{t_j} & \Ext_A^{i+2}(M,N_n)(-f_j) \ar[d]^{y} \\
 			\Ext_A^i(M,N_{n+s})(d_y) \ar[r]^{t_j} & \Ext_A^{i+2}(M,N_{n+s})(d_y - f_j)
 		}
 	\end{equation}
 	for all $ i, n \ge 0 $ and $ 1 \le j \le c $. Thus
 	\begin{equation}\label{bigrad-mod}
 		\mathscr{E}(\mathcal{N}) := \bigoplus_{i \ge 0} \bigoplus_{n \ge 0} \Ext_A^i(M,N_n)
 	\end{equation}
 	turns into a bigraded module over $\mathscr{S}_J := \mathscr{R}(J)[t_1,\ldots,t_c]$, where we set $\deg(t_j) = (2,0)$ for all $1 \le j \le c$, and $\deg(J^s X^s) = (0,s)$ for $s \ge 0$. In \ref{para:Z3-grad-ring}, we show that $ \mathscr{E}(\mathcal{N}) $ has a $ \bZ^3 $-graded module structure.
 \end{para}
 
 \begin{para}\label{para:mod-struc-3}
 	Let $ I $ be a homogeneous ideal of $ A $. Suppose $ J $ is an {\it $ N $-reduction} of $ I $, i.e., $ J \subseteq I $ and $ I^{n+1} N = J I^n N $ for some $ n \ge 0 $. Hence the Rees module $\mathscr{R}(I,N) = \bigoplus_{n \ge 0} I^n N$ is a finitely generated graded module over $\mathscr{R}(J)$. Note that $N[X] = N \otimes_A A[X]$ is a graded module over $A[X]$. Since $\mathscr{R}(J)$ is a graded subring of $A[X]$, we may consider $N[X]$ as a graded $\mathscr{R}(J)$-module whose $ n $th graded component is isomorphic to $ N $ for every $ n \ge 0 $. Therefore, by setting $ \cN := N[X] $, in view of \ref{para:mod-struc-2}, we obtain that
 	\begin{equation}
	 	\mathscr{E}(\mathcal{N}) = \bigoplus_{i \ge 0} \bigoplus_{n \ge 0} \Ext_A^i(M,N)
 	\end{equation}
 	is a bigraded $ \mathscr{S}_J $-module.	Moreover $ \mathcal{L}_I := \bigoplus_{n \ge 0}(N/I^{n+1}N) $ is a graded $ \mathscr{R}(J) $-module, where the graded structure is induced by the exact sequence
 	\begin{equation}\label{bigrad-mod-N}
 		0 \longrightarrow \mathscr{R}(I,N) \longrightarrow N[X] \longrightarrow \mathcal{L}_I(-1) \longrightarrow 0.
 	\end{equation}
 	Thus, by the observations made in \ref{para:mod-struc-2}, we have that
 	\begin{equation}\label{bigrad-mod-N/I}
 		\mathscr{E}(\mathcal{L}_I) = \bigoplus_{i \ge 0} \bigoplus_{n \ge 0} \Ext_A^i(M,N/I^{n+1}N)
 	\end{equation}
 	is a bigraded module over $ \mathscr{S}_J = \mathscr{R}(J)[t_1,\ldots,t_c] $.
 \end{para}
  
 \begin{hypothesis}\label{setup1}
 	Along with Hypothesis~\ref{setup0}, assume that $ \mathcal{N} = \bigoplus_{n \ge 0} N_n $ is a \fg graded $ \mathscr{R}(J) $-module such that $ N_n $ is a $ \bZ $-graded $ A $-module for every $ n \in \bN $.
 	(Note that each $ N_n $ is \fg as an $ A $-module).
 \end{hypothesis}

 One of the main ingredients we use in this article is the following finiteness result.
 
 \begin{theorem}\cite[Thm.~1.1]{Put13}\label{thm:fin-gen}
 	With {\rm Hypothesis~\ref{setup1}},
 	\[
 		\mathscr{E}(\mathcal{N}) = \bigoplus_{i \ge 0} \bigoplus_{n \ge 0} \Ext_A^i(M,N_n)
 	\]
 	is a finitely generated bigraded module over $\mathscr{S}_J = \mathscr{R}(J)[t_1,\ldots,t_c]$.
 \end{theorem}

 \begin{para}[Bigraded modules over standard bigraded ring]\label{para:bigrad-ring}
 	With Hypothesis~\ref{setup1}, we set $ W = \bigoplus_{(i,n) \in \bN^2} W_{(i,n)} := \mathscr{E}(\mathcal{N}) $, i.e., $ W_{(i,n)} := \Ext_A^i(M,N_n) $ for every $ (i,n) \in \bN^2 $. Note that $ \sR(J) $ is a standard $ \bN $-graded $ A $-algebra. Since $ J $ is generated by $ y_1,\ldots,y_b $, we may write $ \sR(J) = A[y_1 X,\ldots,y_bX] $, where $ \deg(A) = 0 $ and $ \deg(y_i X) = 1 $ for $ 1 \le i \le b $. Hence
 	\begin{equation}\label{3.8.1}
 		\sS_J  = \sR(J)[t_1,\ldots,t_c] = A[y_1 X,\ldots,y_bX,t_1,\ldots,t_c]
 	\end{equation}
 	is an $ \bN^2 $-graded ring, where we set $ \deg(A) = (0,0) $, $ \deg(y_i X) = (0,1) $ for $ 1 \le i \le b $, and $ \deg(t_j) = (2,0) $ for $1 \le j \le c$. Note that $ W $ is a finitely generated $ \bN^2 $-graded $ \mathscr{S}_J $-module (by Theorem~\ref{thm:fin-gen}), where each $ W_{(i,n)} $ is a (finitely generated) $ \bZ $-graded $ A $-module. In view of \eqref{3.8.1}, we construct a new standard $ \bN^2 $-graded ring
 	\begin{equation}\label{3.8.2}
 		\sS'_J := \sR(J)[Z_1,\ldots,Z_c] = A[y_1 X,\ldots,y_bX,Z_1,\ldots,Z_c],
 	\end{equation}
 	where $ \deg(A) = (0,0) $, $ \deg(y_i X) = (0,1) $ for $ 1 \le i \le b $, and $ \deg(Z_j) = (1,0) $ for $1 \le j \le c$. Moreover, we set
 	\begin{equation}\label{3.8.3}
 		W^{\rm even} := \bigoplus_{(i,n) \in \bN^2} W_{(2i,n)} \quad \mbox{and} \quad W^{\rm odd} := \bigoplus_{(i,n) \in \bN^2} W_{(2i+1,n)}.
 	\end{equation}
 	Note that $ W_{(2i,n)} $ and $ W_{(2i+1,n)} $ are the $ (i,n) $th graded components of $ W^{\rm even} $ and $ W^{\rm odd} $ respectively. We define the action of $ \sS'_J $ on $ W^{\rm even} $ and $ W^{\rm odd} $ as follows: Elements of $ A[y_1 X,\ldots,y_bX] $ act on $ W^{\rm even} $ and $ W^{\rm odd} $ as before; while the action of $ Z_j $ $ (1 \le j \le c) $ is defined by $	Z_j \cdot m := t_j \cdot m $ for all $ m \in W^{\rm even} $ (resp. $ W^{\rm odd} $). Then, for every $ 1 \le j \le c $, we have
 	\begin{align*}
 	Z_j \cdot W_{(2i,n)} \subseteq W_{(2i+2,n)}, \quad & \mbox{i.e., } Z_j \cdot W^{\rm even}_{(i,n)} \subseteq W^{\rm even}_{(i+1,n)} \mbox{ for all } (i,n) \in \bN^2;\\
 	Z_j \cdot W_{(2i+1,n)} \subseteq W_{(2(i+1)+1,n)}, \quad & \mbox{i.e., } Z_j \cdot W^{\rm odd}_{(i,n)} \subseteq W^{\rm odd}_{(i+1,n)} \mbox{ for all } (i,n) \in \bN^2.
 	\end{align*}
 	Thus $ W^{\rm even} $ and $ W^{\rm odd} $ are $ \bN^2 $-graded $ \sS'_J $-modules. Moreover, $ W^{\rm even} $ and $ W^{\rm odd} $ are finitely generated, which is contained in the proof of \cite[Thm.~3.3.9]{Gho16b}.
 \end{para}

 \begin{para}[Trigraded setup]\label{para:Z3-grad-ring}
 	We make $ \sS'_J $ a $ \bZ^3 $-graded ring as follows. Write
 	\begin{equation}\label{3.9.1}
 		\sS'_J := \sR(J)[Z_1,\ldots,Z_c] = K[x_1,\ldots,x_d, y_1 X,\ldots, y_bX,Z_1, \ldots, Z_c],
 	\end{equation}
 	and set $ \deg(x_i) = (0,0,1) $ for $ 1 \le i \le d $, $ \deg(y_k X) = (0,1,d_k) $ for $ 1 \le k \le b $, and $ \deg(Z_j) = (1,0,-f_j) $ for $ 1 \le j \le c $. Suppose $ V = W^{\rm even} $. We give $ \bZ^3 $-grading structure on $ V $ by setting $ (i,n,a) $th graded component of $ V $ as the $ a $th graded component of the $ \bZ $-graded $ A $-module $ V_{(i,n)} $ for every $ (i,n,a) \in \bZ^3 $. (Note that $ V_{(i,n,a)} = 0 $ if $ i < 0 $ or $ n < 0 $). Thus, in view of \eqref{action} and \ref{para:bigrad-ring}, $ V = W^{\rm even} $ is a $ \bZ^3 $-graded $ \sS'_J $-module. Moreover, $ W^{\rm even} $ is \fg (since we are changing only the grading). In a similar way, one obtains that $ W^{\rm odd} $ is a \fg $ \bZ^3 $-graded $ \sS'_J $-module. In view of \eqref{3.9.1}, we now set a polynomial ring
 	\begin{equation}\label{poly-Z3-grad}
 		S_J := K[X_1,X_2,\ldots,X_d,Y_1,Y_2,\ldots,Y_b,Z_1,Z_2,\ldots,Z_c]
 	\end{equation}
 	over the field $ K $, where $ \deg(X_i) = (0,0,1) $ for $ 1 \le i \le d $, $ \deg(Y_k) = (0,1,d_k) $ for $ 1 \le k \le b $, and $ \deg(Z_j) = (1,0,-f_j) $ for $1 \le j \le c$. We define the action of $ S_J $ on $ W^{\rm even} $ and $ W^{\rm odd} $ as follows: Elements of $ K[Z_1,\ldots,Z_c] $ act on $ W^{\rm even} $ and $ W^{\rm odd} $ as before; while the actions of $ X_i $ and $ Y_j $ (for $ 1 \le i \le d $ and $ 1 \le j \le b $) are defined by $ X_i \cdot m := x_i \cdot m $ and $ Y_j \cdot m := (y_j X) \cdot m $ (respectively) for all $ m \in W^{\rm even} $ (resp. $ W^{\rm odd} $). Thus $ W^{\rm even} $ and $ W^{\rm odd} $ are \fg $ \bZ^3 $-graded $ S_J $-modules.
 \end{para}

 We conclude this section by obtaining the following:
 
 \begin{para}\label{Z3-grad-struc-on-Ext}
 	With Hypothesis~\ref{setup1}, in view of \ref{para:bigrad-ring} and \ref{para:Z3-grad-ring}, we have that
 	\[
 		\mathscr{E}(\mathcal{N})^{\rm even} = \bigoplus_{(i,n) \in \bZ^2} \Ext_A^{2i}(M,N_n) \quad \mbox{and} \quad \mathscr{E}(\mathcal{N})^{\rm odd} = \bigoplus_{(i,n) \in \bZ^2} \Ext_A^{2i+1}(M,N_n)
 	\]
 	are \fg $ \bZ^3 $-graded $ S_J $-modules, where
 	\[
 		\mathscr{E}(\mathcal{N})^{\rm even}_{(i,n,a)} = \Ext_A^{2i}(M,N_n)_a \quad \mbox{and} \quad \mathscr{E}(\mathcal{N})^{\rm odd}_{(i,n,a)} = \Ext_A^{2i+1}(M,N_n)_a
 	\]
 	for every $ (i,n,a) \in \bZ^3 $, and $ S_J $ is a $ \bZ^3 $-graded polynomial ring as in \eqref{poly-Z3-grad}. In view of \ref{reg-over-Q-or-A}, we should note that
 	\[
 		\reg_A\left( \Ext_A^i(M,N_n) \right) = \reg_Q\left( \Ext_A^i(M,N_n) \right) \quad \mbox{for all } i, n \in \bN.
 	\]
 \end{para}
 
 We obtain our main results by proving a theorem in arbitrary trigraded setup; see Theorem~\ref{thm:lin-bdd-multigrad} in the next section.
 
\section{Linear bounds of regularity in multigraded modules}\label{sec: linear bounds in multigraded modules}

 \begin{para}\label{setup1-dipu}
 	We use notations from \ref{notations}. Let
 	\begin{equation}\label{poly-Zt-grad}
 		S := K[X_1,\ldots,X_d,Y_1,\ldots,Y_b,Z_1,\ldots,Z_c]
 	\end{equation}
 	be a polynomial ring over a field $ K $, where we set $ \deg(K) = (0,0,0) $,
 	\begin{align}\label{degrees}
 		& \deg(X_i) := (0,0,1) \; \mbox{ for } 1 \le i \le d; \\
 		& \deg(Y_j) := (0,1,h_j) \; \mbox{ for } 1 \le j \le b \quad \mbox{and} \nonumber \\
 		& \deg(Z_k) := (1,0,g_k) \; \mbox{ for } 1 \le k \le c \nonumber 
 	\end{align}
 	for some $ h_j, g_k \in \bZ $. Note that $ h_j $ and $ g_k  $ are allowed to be negative integers. Without loss of generality, we may assume that
 	\[
 		h_1 \ge h_2 \ge \cdots \ge h_b \quad \mbox{and} \quad g_1 \ge g_2 \ge \cdots \ge g_c.
 	\]
 	Clearly, $ S $ is a Noetherian $ \bZ^3 $-graded ring of dimension $ d' := d + b + c $. Let
 	\[
 		L = \bigoplus_{(i,n,a) \in \bZ^3} L_{(i,n,a)}
 	\]
 	be a finitely generated $ \bZ^3 $-graded $ S $-module. For every $ (i,n) \in \bZ^2 $, we set
 	\begin{equation}\label{S-L-star}
 		S_{(i, n, \star)} := \bigoplus_{a \in \bZ} S_{(i, n, a)} \quad \mbox{and} \quad L_{(i, n, \star)} := \bigoplus_{a \in \bZ} L_{(i, n, a)}.
 	\end{equation}
 	It can be observed that $ S_ {(0,0, \star)} = K[X_1,\ldots,X_d] $ which is a standard $ \bN $-graded polynomial ring  over $ K $, where $ \deg(X_i) = 1 $ for $ 1 \le i \le d $. Set $ Q := S_ {(0,0, \star)} $. Note that $ S = \bigoplus_{(i,n) \in \mathbb{N}^2} S_{(i,n,\star)} $, and
 	\begin{equation}\label{L-over-S}
	 	L = \bigoplus_{(i,n) \in \mathbb{Z}^2} L_{(i,n,\star)} \mbox{ is a $ \bZ^2 $-graded module over } S = \bigoplus_{(i,n) \in \mathbb{N}^2} S_{(i,n,\star)},
 	\end{equation}
 	where $ L_{(i,n,\star)} $ is the $ (i,n) $th graded component of $ L $. Since we are changing only the grading, $ S = \bigoplus_{(i,n) \in \mathbb{N}^2} S_{(i,n,\star)} $ is Noetherian and $ L $ is \fg as a $ \bZ^2 $-graded $ S $-module. Note that $ S $ is standard as an $ \bN^2 $-graded ring, i.e., $ S = S_ {(0,0, \star)}[S_{(1,0,\star)},S_{(0,1,\star)}] $. Thus, for every $ (i,n) \in \bZ^2 $, $ L_{(i,n,\star)} $ is a \fg $ \bZ $-graded module over $ Q = S_ {(0, 0, \star)} $. It can be observed that (for every $ (i,n) \in \bZ^2 $) the assignment $ L \mapsto L_{(i, n, \star)} $ is an exact functor from the category of $ \bZ^3 $-graded $ S $-modules to the category of $ \bZ $-graded $ Q $-modules.
 \end{para}

 \begin{para}\label{setup2-dipu}
 	With the hypotheses as in \ref{setup1-dipu}, by virtue of a multigraded version of Hilbert's Syzygy Theorem, there is a $ \bZ^3 $-graded free resolution
 	\begin{align}
 		& 0 \lra F_{d'} \stackrel{\varphi_{d'}}{\lra} F_{d'-1} \lra \cdots \lra F_1 \stackrel{\varphi_1}{\lra} F_0 \lra 0 \quad \mbox{of $ L $ over $ S $, where} \label{4.2.1} \\
 		& F_l = \bigoplus_p S(-b_{pl1}, -b_{pl2}, -a_{pl}) \quad \mbox{for every } 0 \le l \le d'. \label{4.2.2}
 	\end{align}
 	Note that $ p $ varies over a finite direct sum in $ F_l $. It should be noted that there is no notion of minimal free resolution over arbitrary $ \bZ^3 $-graded ring. We set
 	\begin{align}
 		c_l & := \max_p \left\{ a_{pl} - g_1 b_{pl1} - h_1 b_{pl2} \right\} \mbox{ for } 0 \le l \le d', \mbox{ and} \label{4.2.3} \\
 		e & := \max\{ c_l - l : 0 \le l \le d' \}. \label{4.2.4}
 	\end{align}
 \end{para}
 
 The following result is a multigraded version of \cite[Thm.~1]{Kod00}.
 
 \begin{theorem}\label{thm:lin-bdd-multigrad}
 	With the hypotheses as in {\rm \ref{setup1-dipu} and  \ref{setup2-dipu}}, we have
 	\[
 		\reg_Q\left( L_{(i, n,\star)} \right) \le g_1 \cdot i + h_1 \cdot n + e \quad \mbox{for every } (i, n) \in \bZ^2.
 	\]
 \end{theorem}

\begin{proof}
	Fix an arbitrary $ (i, n) \in \bZ^2 $. In view of \eqref{4.2.1}, the $ (i,n,\star) $th components yield an exact sequence of $ \bZ $-graded $ Q $-modules:
	\begin{equation}\label{4.3.1}
		0 \lra (F_{d'})_{(i,n,\star)} \lra \cdots \lra (F_1)_{(i,n,\star)} \lra (F_0)_{(i,n,\star)} \lra L_{(i,n,\star)} \lra 0.
	\end{equation}
	Note that if $ L_{(i,n,\star)} = 0 $, then there is nothing to prove. So we assume that $ L_{(i,n,\star)} \neq 0 $. Hence $ (F_0)_{(i,n,\star)} \neq 0 $. For every $ 0 \le l \le d' $, we claim that $ (F_l)_{(i,n,\star)} $ is a (finitely generated) $ \bZ $-graded free module over $ Q = K[X_1,\ldots,X_d] $. (Possibly, $ (F_l)_{(i,n,\star)} = 0 $). For every $ a \in \bZ $, it can be observed that
	\begin{align}
		S(-b_{pl1}, -b_{pl2}, -a_{pl})_{(i,n,a)} & = S_{(i - b_{pl1}, n - b_{pl2}, a - a_{pl})} \label{4.3.2}\\
		\cong  \sum_{ \begin{array}{l}
			u_j, v_k \in \bN \;(1 \le j \le b, 1 \le k \le c);\\
			u_1 + u_2 + \cdots + u_b = n - b_{pl2};\\	
			v_1 + v_2 + \cdots + v_c = i - b_{pl1}
			\end{array} } &\left(  \prod_{j = 1}^{b} Y_j^{u_j} \prod_{k = 1}^c Z_k^{v_k} \right) Q_{(a - a_{pl} - \sum u_j h_j - \sum v_k g_k)} \nonumber \\
		\cong \bigoplus_{ \begin{array}{l}
			u_j, v_k \in \bN \;(1 \le j \le b, 1 \le k \le c);\\
			u_1 + u_2 + \cdots + u_b = n - b_{pl2};\\	
			v_1 + v_2 + \cdots + v_c = i - b_{pl1}
			\end{array} } &Q\left(- a_{pl} - \sum_{j=1}^b u_j h_j - \sum_{k = 1}^c v_k g_k \right)_a.\nonumber
	\end{align}
	Note that $ S(-b_{pl1}, -b_{pl2}, -a_{pl})_{(i,n,\star)} = 0 $ if $ i < b_{pl1} $ or $ n < b_{pl2} $. In view of \eqref{4.2.2} and \eqref{4.3.2}, for every $ 0 \le l \le d' $, we have
	\begin{align}
		& (F_l)_{(i,n,\star)} = \bigoplus_p S(-b_{pl1}, -b_{pl2}, -a_{pl})_{(i,n,\star)} \label{4.3.3}\\
		\cong \; \bigoplus_p & \bigoplus_{ \begin{array}{l}
			u_j, v_k \in \bN \;(1 \le j \le b, 1 \le k \le c);\\
			u_1 + u_2 + \cdots + u_b = n - b_{pl2};\\	
			v_1 + v_2 + \cdots + v_c = i - b_{pl1}
			\end{array} } Q\left(- a_{pl} - \sum_{j=1}^b u_j h_j - \sum_{k = 1}^c v_k g_k \right). \nonumber
	\end{align}
	Thus \eqref{4.3.1} provides us a $ \bZ $-graded $ Q $-free resolution of $ L_{(i,n,\star)} $. This resolution is not necessarily minimal. Since $ h_1 \ge \cdots \ge h_b$ and $ g_1 \ge \cdots \ge g_c $, it can be observed from \eqref{4.3.3} that if $ (F_l)_{(i,n,\star)} \neq 0 $, then the maximal twist in $ (F_l)_{(i,n,\star)} $ (as a $ \bZ $-graded $ Q $-free module) is
	\begin{align*}
		&\max \left\{ a_{pl} + (n - b_{pl2}) h_1 + (i - b_{pl1}) g_1 \; \middle| \; \mbox{for all $ p $ such that $ n \ge b_{pl2} $ and $ i \ge b_{pl1} $} \right\} \\
		& \le i g_1 + n h_1 + c_l \quad \mbox{for } 0 \le l \le d' \quad \mbox{[by \eqref{4.2.3}]}.
	\end{align*}
	Hence, in view of \eqref{reg-via-syz} and \eqref{4.2.4}, we obtain that
	\[
		\reg_Q\left( L_{(i, n,\star)} \right) \le g_1 \cdot i + h_1 \cdot n + e \quad \mbox{for every } (i, n) \in \bZ^2.
	\]
	This completes the proof.
\end{proof}

 \begin{remark}
 	In a more general setup, in \cite[Thm.~4.3]{Gho16a}, it is shown that
 	\[
 		\reg \left( L_{(\underline{n},\star)} \right) \le (n_1 + \cdots + n_t) a + b \mbox{ for every } \underline{n} \in \bN^t,
 	\]
 	where $ a $ and $ b $ are some constants. But the method of proof makes it difficult to identify such constants $ a $ and $ b $ explicitly.
 \end{remark}
 
 \begin{remark}
 	Similar proof as in Theorem~\ref{thm:lin-bdd-multigrad} works for arbitrary multigraded setup. In analogous multigraded setup, we obtain that
 	\[
 		\reg_Q\left( L_{(\underline{n},\star)} \right) \le n_1 d_1 + \cdots + n_t d_t + e \quad \mbox{for every } \underline{n} \in \bZ^t,
 	\]
 	where the coefficients $ d_1,\ldots,d_t $ can be determined explicitly in terms of degrees of the indeterminates of $ S $, and the constant term $ e $ can be obtained from a multigraded free resolution of $ L $ over $ S $ (as in \ref{setup2-dipu}).
 \end{remark}
%

 As a consequence of Theorem~\ref{thm:lin-bdd-multigrad}, we obtain the following:
 
 \begin{corollary}\label{cor:lin-bdd-ext}
 	With {\rm Hypothesis~\ref{setup1}}, there are constants $ e_0 $ and $ e_1 $ such that
 	\begin{enumerate}[{\rm (i)}]
 		\item $ \reg\left( \Ext_A^{2i}(M, N_n) \right) \le d_1 \cdot n - f_1 \cdot i + e_0 $ for all $ i, n \ge 0 $.
 		\item $ \reg\left( \Ext_A^{2i+1}(M, N_n) \right) \le d_1 \cdot n - f_1 \cdot i + e_1 $ for all $ i, n \ge 0 $.
 	\end{enumerate}
 \end{corollary}

\begin{proof}
	In view of \ref{Z3-grad-struc-on-Ext}, $ \mathscr{E}(\mathcal{N})^{\rm even} $ and $ \mathscr{E}(\mathcal{N})^{\rm odd} $ are \fg $ \bZ^3 $-graded $ S_J $-modules, where
	\[
		S_J := K[X_1,X_2,\ldots,X_d,Y_1,Y_2,\ldots,Y_b,Z_1,Z_2,\ldots,Z_c]
	\]
	with $ \deg(X_i) = (0,0,1) $ for $ 1 \le i \le d $, $ \deg(Y_k) = (0,1,d_k) $ for $ 1 \le k \le b $, and $ \deg(Z_j) = (1,0,-f_j) $ for $1 \le j \le c$. Note that $ d_1 \ge \cdots \ge d_b $ and $ - f_1 \ge \cdots \ge - f_c $. Moreover, for every $ (i,n) \in \bN^2 $, we have
	\[
		\mathscr{E}(\mathcal{N})^{\rm even}_{(i,n,\star)} = \Ext_A^{2i}(M, N_n) \quad \mbox{and} \quad \mathscr{E}(\mathcal{N})^{\rm odd}_{(i,n,\star)} = \Ext_A^{2i+1}(M, N_n).
	\]
	Therefore the result follows by applying Theorem~\ref{thm:lin-bdd-multigrad} to $ \mathscr{E}(\mathcal{N})^{\rm even} $ and $ \mathscr{E}(\mathcal{N})^{\rm odd} $.
\end{proof}

\section{Linear bounds of regularity for certain Exts}\label{sec: regularity of Exts}
 
 We are now in a position to obtain our main results. We prove these results with the following setup.
 
 \begin{hypothesis}\label{main-hyp}
 	Let $ Q = K[X_1,\ldots,X_d] $ be a polynomial ring over a field $ K $, where $ \deg(X_i) = 1 $ for $ 1 \le i \le d $. Set $ A := Q/({\bf z}) $, where $ {\bf z}  = z_1,\ldots,z_c $ is a homogeneous $ Q $-regular sequence. Let $ I $ be a homogeneous ideal of $ A $. Let $ M $ and $ N $ be \fg $ \bZ $-graded $ A $-modules. Set $ f := \min\{ \deg(z_j) : 1 \le j \le c \} $.
 \end{hypothesis}

\begin{para}\label{para:rho_N(I)}
	Recall that an ideal $ J $ of $ A $ is said to be an $ N $-{\it reduction} of $ I $ if $ J \subseteq I $ and $ I^{n+1}N = J I^n N $ for some $ n \ge 0 $. The invariant $ \rho_N(I) $ is defined by
	\begin{equation}\label{rho_N(I)}
		\rho_N(I) := \min\{ d(J) : J \mbox{ is a homogeneous $ N $-reduction ideal of $ I $} \},
	\end{equation}
	where $ d(J) $ denotes the minimum number $ n $ such that $ J $ is generated by homogeneous elements of degree $ \le n $. 
\end{para}

\begin{para}\label{para:filtration}
	A sequence of ideals $ \{ I_n \}_{n \ge 0} $ is called an $ I $-{\it filtration} if $ I I_n \subseteq I_{n+1} $ for all $ n \ge 0 $. In addition, if $ I I_n = I_{n+1} $ for all sufficiently large $ n $, then $ \{ I_n \}_{n \ge 0} $ is said to be a {\it stable $ I $-filtration}.
\end{para}

Here are our main results.

\begin{theorem}\label{thm:main}
	Along with {\rm Hypothesis~\ref{main-hyp}}, further assume that $ \{I_n\}_{n \ge 0} $ is a stable $ I $-filtration. Then, for every $ l \in \{0,1\} $, there exist constants $ e_{l1} $ and $ e_{l2} $ such that the following inequalities hold true.
	\begin{enumerate}[{\rm (i)}]
		\item $ \reg\left( \Ext_A^{2i+l}(M, I_n N) \right) \le \rho_N(I) \cdot n - f \cdot i + e_{l1} $ for all $ i, n \ge 0 $.
		\item $ \reg\left( \Ext_A^{2i+l}(M,N/I_n N) \right) \le \rho_N(I) \cdot n - f \cdot i + e_{l2} $ for all $ i, n \ge 0 $.
	\end{enumerate}
	The particular case when $ I_n = I^n $ proves the statements given in the introduction.
\end{theorem}

\begin{proof}
	(i) Let $ J $ be a homogeneous $ N $-reduction ideal of $ I $ such that $ \rho_N(I) = d(J) $. Then $ J \subseteq I $ and $ I^{n+1}N = J I^n N $ for all $ n \ge n_0 $ for some $ n_0 \in \bN $. Hence
	\begin{equation}\label{eqn1-thm:main}
		I^n N = J^{n - n_0} I^{n_0} N \quad \mbox{for all } n \ge n_0.
	\end{equation}
	Suppose $ J $ is minimally generated by homogeneous elements of degree $ d_1,\ldots,d_b $ such that $ d_1 \ge \cdots \ge d_b $. Note that $ \rho_N(I) = d(J) = d_1 $. Since $ \{I_n\}_{n \ge 0} $ is a stable $ I $-filtration, $ J^m \left( I_n N \right) \subseteq I^m I_n N \subseteq I_{m+n} N $ for all $ m , n \in \bN $. Hence $ \cN := \bigoplus_{n \ge 0} I_n  N $ is a graded $ \sR(J) $-module. Moreover, there is $ n_1 $ such that $ I_n  = I^{n-n_1} I_{n_1} $ for all $ n \ge n_1 $. So, by \eqref{eqn1-thm:main}, we have
	\[
		I_n N = I^{n-n_1} I_{n_1} N = J^{n - n_0 - n_1} \left( I^{n_0} I_{n_1} N \right)
	\]
	for all $ n \ge n_0 + n_1 $, which yields that $ \cN = \bigoplus_{n \ge 0} I_n  N $ is a \fg graded $ \sR(J) $-module. Hence the inequality follows from Corollary~\ref{cor:lin-bdd-ext}.
	
	(ii) We show that there is some constant $ e_{02} $ such that
	\begin{equation}\label{eqn2-thm:main}
		\reg\left( \Ext_A^{2i}(M,N/I_nN) \right) \le \rho_N(I) \cdot n - f \cdot i + e_{02} \quad \mbox{for all }  i, n \ge 0.
	\end{equation}
	In a similar way, one can prove that there is some constant $ e_{12} $ such that
	\begin{equation*}
	\reg\left( \Ext_A^{2i+1}(M,N/I_nN) \right) \le \rho_N(I) \cdot n - f \cdot i + e_{12} \quad \mbox{for all }  i, n \ge 0.
	\end{equation*}
	To show \eqref{eqn2-thm:main}, we consider the exact sequences $ 0 \to I_n N \to N \to N/I_n N \to 0 $ for $ n \in \mathbb{N} $. These yield exact sequences of graded $ A $-modules:
	\begin{equation}\label{eqn3-thm:main}
		\Ext_A^i(M,I_n N) \rightarrow \Ext_A^i(M,N) \rightarrow \Ext_A^i(M,N/I_n N)	\rightarrow \Ext_A^{i+1}(M,I_n N).
	\end{equation}
	In view of \ref{para:mod-struc-3}, taking direct sum over $ i, n $ in \eqref{eqn3-thm:main}, and using the naturality of the Eisenbud operators $ t_j $, we obtain an exact sequence of bigraded $ \sS_J  = \sR(J)[t_1,\ldots,t_c] $-modules:
	\begin{align}\label{eqn4-thm:main}
		W \lra U \lra V \lra W(1,0), \quad \mbox{where} \quad W := \bigoplus_{i, n \ge 0} \Ext_A^i(M,I_n N),\\
		U := \bigoplus_{i, n \ge 0} \Ext_A^i(M, N) \quad \mbox{and} \quad V := \bigoplus_{i, n \ge 0} \Ext_A^i\left(M,\frac{N}{I_n N}\right). \nonumber
	\end{align}
	Note that $ W(1,0) $ is same as $ W $ but the grading is twisted by $ (1,0) $. Setting
	\[
		C := \Image(W \to U), ~ D := \Image(U \to V) \; \mbox{ and } \; E := \Image(V \to W(1,0)),
	\]
	we have the following short exact sequences of bigraded $ \sS_J $-modules:
	\begin{equation}\label{eqn5-thm:main}
		0 \lra C \lra U \lra D \lra 0 \quad \mbox{ and } \quad 0 \lra D \lra V \lra E \lra 0.
	\end{equation}
	Since $ W $ is a \fg $ \sS_J $-module, $ C $ and $ E $ are also \fg $ \sS_J $-modules. Moreover, in view of \ref{para:bigrad-ring} and \ref{para:Z3-grad-ring},
	\[
		C^{\rm even} = \bigoplus_{(i,n) \in \bN^2} C_{(2i,n)} \quad \mbox{ and } \quad E^{\rm even} = \bigoplus_{(i,n) \in \bN^2} E_{(2i,n)}
	\]
	are \fg $ \bZ^3 $-graded $ S_J $-modules, where
	\[
		C^{\rm even}_{(i,n,\star)} = C_{(2i,n)} \quad \mbox{ and } \quad E^{\rm even}_{(i,n,\star)} = E_{(2i,n)} \quad \mbox{for } (i,n) \in \bN^2.
	\]
	Hence, by virtue of Theorem~\ref{thm:lin-bdd-multigrad}, there exist some constants $ e_C $ and $ e_E $ such that
	\begin{align}
		&\reg\left(C_{(2i,n)}\right) \le \rho_N(I) \cdot n - f \cdot i + e_C \quad \mbox{for all } i, n \ge 0, \label{eqn6-thm:main} \\
		&\reg\left(E_{(2i,n)}\right) \le \rho_N(I) \cdot n - f \cdot i + e_E  \quad \mbox{for all } i, n \ge 0. \label{eqn7-thm:main}
	\end{align}
	The sequences \eqref{eqn5-thm:main} yield exact sequences of \fg graded $ A $-modules:
	\begin{align}
		& 0 \lra C_{(2i,n)} \lra U_{(2i,n)} \lra D_{(2i,n)} \lra 0 \quad \mbox{ and}\label{eqn8-thm:main} \\
		& 0 \lra D_{(2i,n)} \lra V_{(2i,n)} \lra E_{(2i,n)} \lra 0 \quad \mbox{for } (i,n) \in \bN^2. \label{eqn9-thm:main}
	\end{align}
	Considering the improper ideal $ A $ in (i), since $ \rho_N(A) = 0 $, there is some constant $ e_A $ such that $ \reg\left( \Ext_A^{2i}(M, N) \right) \le - f \cdot i + e_A $, i.e., 
	\begin{equation}\label{eqn10-thm:main}
		\reg\left( U_{(2i,n)} \right) \le - f \cdot i + e_A \mbox{ for all } i, n \ge 0.
	\end{equation}
	In view of \eqref{eqn8-thm:main}, by virtue of Lemma~\ref{lemma: properties of regularity}(iii), we get that
	\begin{align}
		\reg\left( D_{(2i,n)} \right) & \le \max\left\{ \reg\left( C_{(2i,n)} \right) - 1, \reg\left( U_{(2i,n)} \right) \right\} \label{eqn11-thm:main} \\
		& \le \rho_N(I) \cdot n - f \cdot i + e_D \quad \mbox{[by \eqref{eqn6-thm:main} and \eqref{eqn10-thm:main}]} \nonumber
	\end{align}
	for all $ i, n \ge 0 $, where $ e_D := \max\{ e_C - 1, e_A \} $. Hence
	\begin{align*}
		\reg & \left( \Ext_A^{2i}(M,N/I_n N) \right) = \reg\left( V_{(2i,n)} \right) \\
		& \le \max\left\{ \reg\left( D_{(2i,n)} \right), \reg\left( E_{(2i,n)} \right) \right\} \quad \mbox{[by \eqref{eqn9-thm:main} and Lemma~\ref{lemma: properties of regularity}(ii)]} \\
		& \le \rho_N(I) \cdot n - f \cdot i + e_{02} \quad \mbox{[by \eqref{eqn7-thm:main} and \eqref{eqn11-thm:main}]}
	\end{align*}
	for all $ i, n \ge 0 $, where $ e_{02} := \max\{e_D,e_E\} $. This completes the proof of (ii).
\end{proof}

 \section{Examples}\label{sec:examples}
 
 The following examples show that the bounds in Theorem~\ref{thm:main} are sharp. More precisely, the leading coefficients of the functions
 \[
 	\reg\big( \Ext_A^{2i+l}(M,I^nN) \big) \quad \mbox{and} \quad  \reg\big( \Ext_A^{2i+l}(M,N/I^nN) \big) \quad \mbox{for } l \in \{0,1\}
 \]
 can be exactly $ \rho_N(I) $ and $ - f $.

\begin{example}\label{example1}
	Let $ A = K[X]/(X^2) $, where $ K $ is a field and $ \deg(X) = 1 $. Suppose $ x $ is the residue of $ X $ in $ A $. Set $ M = N := A/(x) $, and $ I := A $. Then
	\[
		\reg\left( \Ext_A^{2i}(M,I^nN) \right) = - 2i \quad \mbox{and} \quad \reg\left( \Ext_A^{2i+1}(M,I^nN) \right) = - 2i - 1
	\]
	for all $ i, n \ge 0 $. In this case, note that $ \rho_N(I) = 0 $ and $ f = \deg(X^2) = 2 $.
\end{example}

\begin{proof}
	Consider the minimal graded free resolution
	\[
		\mathbb{F}_M : \quad \cdots \lra A(-3) \stackrel{x}{\lra} A(-2) \stackrel{x}{\lra} A(-1) \stackrel{x}{\lra} A \lra 0
	\]
	of $ M $ over $ A $. Since $ \Hom_A(A(-i),N) \cong N(i) $ for every $ i \ge 0 $, the complex $ \Hom_A(\mathbb{F}_M,N) $ is isomorphic to
	\[
		0 \lra N \stackrel{x}{\lra} N(1) \stackrel{x}{\lra} N(2) \stackrel{x}{\lra} N(3) \lra \cdots.
	\]
	Since $ x $ annihilates $ N $, the map $ N(i) \stackrel{x}{\lra} N(i+1) $ is a zero map for every $ i \ge 0 $. Hence $ \Ext_A^i(M,N) = N(i) $ for every $ i \ge 0 $. Therefore it follows from the fact
	\[
		N(i)_{\mu} = \left\{ \begin{array}{ll}
		0 & \mbox{if } \mu \neq - i \\
		N_0 = K \neq 0 & \mbox{if } \mu = -i \end{array} \right.
	\]
	that $ \reg\left( \Ext_A^i(M,N) \right) = - i $ for every $ i \ge 0 $.
\end{proof}
%

\begin{example}\label{example2}
	Let $ A = K[X,Y]/(X^2,Y^3) $, where $ K $ is a field and $ \deg(X) = \deg(Y) = 1 $. Suppose $ x, y $ are the residues of $ X,Y $ in $ A $ respectively. Set $ M = N := A/(y) $, and $ I := A $. Then
	\[
		\reg\left( \Ext_A^{2i}(M,I^nN) \right) = - 3i + 1 \quad \mbox{and} \quad \reg\left( \Ext_A^{2i+1}(M,I^nN) \right) = - 3 i
	\]
	for all $ i, n \ge 0 $. Note that $ \rho_N(I) = 0 $ and $ f = \min\{ \deg(X^2), \deg(Y^3) \} = 2 $.
\end{example}

\begin{proof}
	Consider the minimal graded free resolution $ \mathbb{F}_M : $
	\[
		\cdots \lra A(-7) \stackrel{y}{\lra} A(-6) \stackrel{y^2}{\lra} A(-4) \stackrel{y}{\lra} A(-3) \stackrel{y^2}{\lra} A(-1) \stackrel{y}{\lra} A \lra 0
	\]
	of $ M $ over $ A $. The complex $ \Hom_A(\mathbb{F}_M,N) $ is isomorphic to
	\[
		0 \lra N \stackrel{y}{\lra} N(1) \stackrel{y^2}{\lra} N(3) \stackrel{y}{\lra} N(4) \stackrel{y^2}{\lra} N(6) \stackrel{y}{\lra} N(7) \lra \cdots.
	\]
	Since $ y $ annihilates $ N $, all the maps in $ \Hom_A(\mathbb{F}_M,N) $ are zero maps. Hence $ \Ext_A^{2i}(M,N) = N(3i) $ and $ \Ext_A^{2i+1}(M,N) = N(3i+1) $ for every $ i \ge 0 $. Since
	\[
		N(m) \cong (K \oplus Kx)(m) \cong (K \oplus K(-1))(m) \cong K(m) \oplus K(m-1),
	\]
	$ \reg(N(m)) = - m + 1 $. Hence the assertion follows.
\end{proof}

\begin{remark}\label{example2a}
	With $ A $, $ M $ and $ N $ as in Example~\ref{example2}, one can easily compute that $ \Tor_{2i}^A(M,N) = N(-3i) $ and $ \Tor_{2i+1}^A(M,N) = N(-3i-1) $ for all $ i \ge 0 $. Hence
	\[
		\reg\left( \Tor_{2i}^A(M,N) \right) = 3i + 1 \quad \mbox{and} \quad \reg\left( \Tor_{2i+1}^A(M,N) \right) = 3i + 2 \quad \mbox{for all $ i \ge 0 $}.
	\]
	By setting $ U = V := A/(x) $, it can be easily obtained that $ \Tor_i^A(U,V) = V(-i) $ for every $ i \ge 0 $. Since
	\[
		V(-i) \cong (K \oplus K y \oplus K y^2)(-i) \cong K(-i) \oplus K(-i-1) \oplus K(-i-2),
	\]
	we have $ \reg(\Tor_i^A(U,V)) = \reg(V(-i)) = i + 2 $ for every $ i \ge 0 $. Therefore
	\[
		\reg\left( \Tor_{2i}^A(U,V) \right) = 2i + 2 \quad \mbox{and} \quad \reg\left( \Tor_{2i+1}^A(U,V) \right) = 2i + 3 \quad \mbox{for all $ i \ge 0 $}.
	\]
\end{remark}

\begin{example}\label{example3}
	Assume $ A = K[X,Y]/(XY) $, where $ K $ is a field, and $ \deg(X) = \deg(Y) = 1 $. Suppose $ x, y $ are the residues of $ X,Y $ in $ A $ respectively. Set $ M := A/(x) $, $ N := (x) $ and $ I := (x) $. Then
	\begin{enumerate}[(i)]
		\item $ \reg\left( \Ext_A^{2i}(M,I^nN) \right) = - \infty $ and $ \reg\left( \Ext_A^{2i+1}(M,I^nN) \right) = n - 2i $ for $ i, n \ge 0 $.
		\item $ \reg\left( \Ext_A^{2i}(M,N/I^nN) \right) = n - 2i $ and $ \reg\left( \Ext_A^{2i+1}(M,N/I^nN) \right) = - 2i $ for all $ i \ge 0 $ and $ n \ge 1 $.
	\end{enumerate}
	In this case, note that $ \rho_N(I) = 1 $ and $ f = \deg(XY) = 2 $.
\end{example} 	

\begin{proof}
	Consider the minimal graded free resolution
	\[
		\mathbb{F}_M : \quad \cdots \stackrel{y}{\lra} A(-3) \stackrel{x}{\lra} A(-2) \stackrel{y}{\lra} A(-1) \stackrel{x}{\lra} A \lra 0
	\]
	of $ M $ over $ A $. For a graded $ A $-module $ L $, the complex $ \Hom_A(\mathbb{F}_M, L) $ is isomorphic to
	\[
		0 \lra L \stackrel{x}{\lra} L(1) \stackrel{y}{\lra} L(2) \stackrel{x}{\lra} L(3) \stackrel{y}{\lra} \cdots.
	\]
	
	(i) Suppose $ L = I^n N $. Then $ L = (x^{n+1}) $. Moreover,
	\begin{align*}
		&\Ker\left( L(2i) \stackrel{x}{\lra} L(2i+1) \right) = 0, \quad \Image\left( L(2i) \stackrel{x}{\lra} L(2i+1) \right) = (xL)(2i+1) \\
		&\mbox{and} \quad \Ker\left( L(2i+1) \stackrel{y}{\lra} L(2i+2) \right) = L(2i+1) \quad \mbox{for every } i \ge 0.
	\end{align*}
	Therefore $ \Ext_A^{2i}(M,I^nN) = 0 $ and
	 \begin{align*}
	 	\Ext_A^{2i+1}(M,I^nN) & = \frac{L(2i+1)}{(xL)(2i+1)} \cong \left(\frac{L}{xL}\right)(2i+1) \cong \left( K x^{n+1} \right)(2i+1) \\
	 	& \cong K(-n-1)(2i+1) = K(-n+2i)
	 \end{align*}
	 for all $ i,n \ge 0 $. Hence the assertion (i) follows.
	 
	 (ii) Let $ W = N/I^n N $. Then $ W = (x)/(x^{n+1}) \cong Kx \oplus Kx^2 \oplus \cdots \oplus Kx^n $. Moreover,
	 \begin{align*}
	 &\Ker\left( W(2i) \stackrel{x}{\lra} W(2i+1) \right) = (Kx^n)(2i) \cong K(-n+2i),\\
	 & \Image\left( W(2i) \stackrel{x}{\lra} W(2i+1) \right) = (xW)(2i+1) \cong (Kx^2 \oplus \cdots \oplus Kx^n)(2i+1) \\
	 &\mbox{and} \quad \Ker\left( W(2i+1) \stackrel{y}{\lra} W(2i+2) \right) = W(2i+1) \quad \mbox{for every } i \ge 0.
	 \end{align*}
	 Therefore $ \Ext_A^{2i}(M,N/I^nN) \cong K(-n+2i) $ and
	 \begin{equation*}
	 	\Ext_A^{2i+1}(M,N/I^nN) = \frac{W(2i+1)}{(xW)(2i+1)} \cong (Kx)(2i+1) \cong K(2i)
	 \end{equation*}
	 for all $ i \ge 0 $ and $ n \ge 1 $. Hence the assertion (ii) follows.
\end{proof}

\begin{remark}\label{example4b}
	With $ A $, $ M $ and $ N $ as in Example~\ref{example3}, we have
	\[
		\reg\left( \Tor_{2i}^A(M,N) \right) = 2i+1 \quad \mbox{and} \quad \reg\left( \Tor_{2i+1}^A(M,N) \right) = 0 \quad \mbox{for all $ i \ge 0 $}.
	\]
\end{remark}

\begin{proof}
	Considering the minimal graded free resolution $ \mathbb{F}_M $ of $ M $ over $ A $, we have
	\[
		\mathbb{F}_M \otimes_A N : \quad \cdots \stackrel{y}{\lra} N(-3) \stackrel{x}{\lra} N(-2) \stackrel{y}{\lra} N(-1) \stackrel{x}{\lra} N \lra 0.
	\]
	Therefore, for every $ i \ge 0 $,
	\begin{align*}
		\Tor_{2i}^A(M,N) & =  \dfrac{\Ker(N(-2i) \stackrel{y}{\lra} N(-2i+1))}{\Image( N(-2i-1) \stackrel{x}{\lra} N(-2i) )} = \dfrac{N(-2i)}{(xN)(-2i)} \\
		& \cong (Kx)(-2i) \cong K(-2i-1) \\
	   	\mbox{and} \quad \Tor_{2i+1}^A(M,N) & = \Ker\left( N(-2i-1) \stackrel{x}{\lra} N(-2i) \right) = 0.
	\end{align*}
	Hence the assertion follows.
\end{proof}

We conclude our paper by stating the following natural question.

\begin{question}
	For $ \ell \in \{0,1\} $, do there exist $ a_\ell , a'_\ell \in \mathbb{Z}_{>0} $ and $ e_\ell , e'_\ell \in \mathbb{Z} \cup \{ -\infty \} $ such that
	\begin{enumerate}[{\rm (i)}]
		\item $ \reg\left(\Ext_A^{2i+\ell}(M,N)\right) = - a_\ell \cdot i + e_\ell $ for all $ i \gg 0 $ ?
		\item $ \reg\left(\Tor_{2i+\ell}^A(M,N)\right) = a'_\ell \cdot i + e'_\ell $ for all $ i \gg 0 $ ?
	\end{enumerate}
\end{question}

%

\section*{Acknowledgements}
Ghosh was supported by DST, Government of India under the DST-INSPIRE Faculty Scheme with Award Reg. No. DST/INSPIRE/04/2016/000587.

\end{document}